\documentclass{amsart}
\usepackage{enumerate}
\usepackage[dvips]{graphicx}
\usepackage{color}
\usepackage{pstricks}
\usepackage{pst-node}
\newtheorem{theo}{Theorem}[section]
\newtheorem{cor}[theo]{Corollary}
\newtheorem{lem}[theo]{Lemma}

\newtheorem{prop}[theo]{Proposition}

\newtheorem{defn}[theo]{Definition}

\theoremstyle{definition}

\def\a{{\a}}
\def\a{{\mathfrak a}}

\def\C{\mathbb C}

\def\CC{\mathcal C}

\def\H{\mathbb H}

\def\N{\mathbb N}
\def\P{\mathbb{P}}

\def\R{\mathbb R}

\def\Z{\mathbb{Z}}
\begin{document} 

\title[Interacting particle models and  the Pieri-type formulas]{Interacting particle models and  the Pieri-type formulas :\\ the symplectic case with non equal weights}

\author{Manon Defosseux}
\address{Laboratoire de Math\'ematiques Appliqu\'ees \`a Paris 5, Universit\'e Paris 5, 45 rue des  Saints P\`eres, 75270 Paris Cedex 06.}
\email{manon.defosseux@parisdescartes.fr}
%    General info
%\date{a}

\begin{abstract} We have introduced in \cite{DefosseuxSO} a particle model with  blocking and pushing interactions related to a Pieri type formula for the orthogonal group. A symplectic version of this model is presented here. It leads in particular to the particle model with a wall defined in \cite{WarrenWindridge}.    
\end{abstract}
\maketitle  
\section{introduction}
Let us recall that the Pieri's formula describes the product  of a Schur polynomial by a complete symmetric function.  It is a specific case of the Littlewood-Richardson rules for decomposing the tensor products  of  representations of the unitary group  into irreducible components.  In that specific case, irreducible components of the tensor products have a multiplicity equal to one.

 In \cite{DefosseuxSO} an interacting particle model has been introduced and proved  to be  related to a Pieri-type formula for the orthogonal group. Particles of the model can move to the left or to the right and are submitted to  blocking and pushing interactions. Moreover, they are constrained to stay  on the non-negative real axis. Such a model is said to be with a wall.  One can find for instance in \cite{BorodinKuan} and \cite{WarrenWindridge} other examples of models with a wall. In these last two references, models only differ by the behavior of the particles near the wall : in the first one, these particles are reflected by the wall, in the second one, they are blocked by the wall.  Actually the first one,  connected to models of   \cite{DefosseuxSO},  is strongly related to representations  of the orthogonal group whereas the second one involves  representations of the symplectic compact group. Pursuing the study of models with a wall  in the same way as in  \cite{DefosseuxSO}  we construct  here  a new interacting particle model with a wall depending on parameters. Two particular values of the parameters lead on the one hand to a model studied in \cite{WarrenWindridge} on the other hand to a random matrix model of \cite{Defosseux}. 
 
 A fundamental difference between models considered  here  and models usually considered in the literature is that they are related to tensor products of representations which aren't decomposed into irreducible components with multiplicity one.  Besides, the case with non equal weights  studied  here is of particular interest since it exposes connections between symplectic Schur functions and the random particle models. For these reasons, this paper could hopefully  provide a step towards a generalization.

The paper is organized as follows. In the second section, I shall recall the definitions of the symplectic Gelfand-Tsetlin patterns and Schur functions. In the third one I will describe the interacting particle model studied in the paper. In the fourth section I'll recall some properties about tensor product of particular representations  of the symplectic compact group, which naturally leads to some Markov kernels involved in the interacting particle model. These Markov kernels are defined in the fifth section. Section six is devoted to a random matrix model related to the particle model. Results of the paper are stated in section seven. I will  sketch the proofs in the last section.  

\section{Symplectic Gelfand-Tsetlin patterns, symplectic Schur functions}\label{definitions}
Here and elsewhere $\mathbb{N}$ stands for the set of nonnegative integers. For  $n\in \N^*$ and $x, y\in \mathbb{R}^{n}$ such that $x_n\le\dots\le x_1$ and $y_n\le\dots\le y_1$, we write $x \preceq y$ if $x$ and $y$ are interlaced, i.e.\@
$$x_{n}\le y_n\le x_{n-1} \le \cdots\le x_1 \le y_1 .$$
When $x\in \mathbb{R}^{n}$ and $y\in\mathbb{R}^{n+1}$ we add the relation $y_{n+1} \leq x_{n}$. We  denote by $\vert x\vert$ the sum of the coordinates $\sum_{i=1}^n x_i$. 
\begin{defn}  Let $k$ be a positive integer.
\begin{enumerate}  
\item We denote by $GT_{k}$  the set of symplectic Gelfand-Tsetlin patterns defined by  
\begin{eqnarray*}
GT_{k}=\{(x^{1},\cdots, x^{k}): x^{i}\in \mathbb{N}^{[\frac{i+1}{2}]} \mbox{ and }   x^{i-1}   \preceq    x^{i}  , 1\leq i\leq k\}.
\end{eqnarray*}
  
 \item If $x=(x^{1},\dots,x^{k})$ is a symplectic  Gelfand-Tsetlin pattern, $x^{i}$ is called the $i^{th}$ row of $x$ for $i\in\{1,\dots,k\}$.
  \item For $\lambda\in \N^{[\frac{k+1}{2}]}$  such that 
  $\lambda_{[\frac{k+1}{2}]}\le \dots \le \lambda_1,$ the subset of $GT_k$ of symplectic  Gelfand-Tsetlin patterns    having a $k^{th}$ row equal to $\lambda$ is denoted by $GT_k(\lambda)$.
 \end{enumerate}
\end{defn}  
 Usually, a symplectic  Gelfand-Tsetlin pattern is represented by a triangular array as indicated at figure \ref{coneC} for $k=2r$.  Actually the more right an entry is, the largest it is. Thus  the interlacing property satisfied by the coordinates of a symplectic  Gelfand-Tsetlin pattern  is graphically  represented. Note that the zero and the minus indicate that the coordinates are positive. For $x$ a symplectic  Gelfand-Tsetlin pattern of $GT_k$ and $(q_i)_{i\ge 1}$ a sequence of   positive real numbers  one defines $w^{k}_x(q_1,\dots,q_{[\frac{k+1}{2}]})$ by recursion by letting 
 $$   w^{1}_x(q_1)=q_1^{\vert x^1\vert},$$ and
 \begin{align*}
      w^{2i}_x(q_1,\dots,q_i)&=w_x^{2i-1}(q_1,\dots,q_i) \,q_i^{\vert x^{2i-1}\vert-\vert x^{2i}\vert} , \\
      w^{2i+1}_x(q_1,\dots,q_i,q_{i+1})&=w_x^{2i}(q_1,\dots,q_i)\,q_{i+1}^{\vert x^{2i+1}\vert-\vert x^{2i}\vert},
\end{align*}
 for $i\in\N^*.$
 \begin{defn}   For $\lambda\in \N^{[\frac{k+1}{2}]}$ such that 
 $\lambda_1\ge \dots\ge \lambda_{[\frac{k+1}{2}]},$  we denote by $s^k_\lambda$ the symplectic Schur function defined by
 $$s^k_\lambda(q)=\sum_{x\in GT_k(\lambda)}w_x^k(q),$$
 for $q=(q_1,\dots,q_{[\frac{k+1}{2}]})\in \R_+^{[\frac{k+1}{2}]}$.
 \end{defn} 
Notice that  the   cardinality of $GT_k(\lambda)$ is equal to $s_\lambda^k(\bf{1})$, with ${\bf 1}=(1,\dots,1)\in \R^{[\frac{k+1}{2}]}$.

\begin{figure}[h!]
\begin{pspicture}(2,3.2)(9,-0.3)

 \put(0.25,-0.5){$-x^{2r}_{1}$}  \put(2,-0.5){$\cdots$}  \put(4,-0.5){$-x^{2r}_{r}$} \put(5.35,-0.5){$0$}    \put(6.25,-0.5){$x^{2r}_{r}$}  \put(8.5,-0.5){$\cdots $}  \put(10.25,-0.5){$x^{2r}_{1}$}
 
 \put(1,0){$-x^{2r-1}_{1}$}  \put(2.75,0){$\cdots$}  \put(4.5,0){$-x^{2r-1}_{r}$}  \put(5.75,0){$x^{2r- 1}_{r}$}  \put(7.75,0){$\cdots $}  \put(9.5,0){$x^{2r-1}_{1}$}
 \put(4.25,0.5){$\cdots$}  \put(6,0.5){$\cdots$} 
   \put(2.25,1){$-x^{4}_{1}$}     \put(3.75,1){$-x^{4}_{2}$}     \put(5.35,1){$0$}   \put(6.5,1){$x^{4}_{2}$}   \put(8,1){$x^{4}_{1}$}  
     \put(3,1.5){$-x_{1}^{3}$}  \put(4.5,1.5){$-x_{2}^{3}$} \put(5.75,1.5){$x_{2}^{3}$} \put(7.125,1.5){$x_{1}^{3}$} 
       \put(3.75,2){$-x^{2}_{1}$}       \put(5.35,2){$0$}   \put(6.5,2){$x^{2}_{1}$}    
      \put(4.5,2.5){$-x^{1}_{1}$}       \put(5.75,2.5){$x^{1}_{1}$} 
    \put(5.35,3){$0$}    

\end{pspicture}
  \caption{A symplectic  Gelfand--Tsetlin pattern of   $GT_{2r}$}
  \label{coneC}
\end{figure} 
 
\section{Interacting particle models} 
In this section we  construct two processes   evolving on the set $GT_k$ of symplectic Gelfand-Tsetlin patterns. These processes can be viewed as   interacting particle models. For this, we associate to a symplectic Gelfand-Tsetlin pattern $x=(x^{1},\dots,x^{k})$, a configuration of particles on the integer lattice $\Z^2$ putting one particle labeled by $(i,j)$ at point $(x_j^{i},k-i)$ of $\Z^2$ for $i\in\{1,\dots,k\}$, $j\in\{1,\dots,[\frac{i+1}{2}]\}$. Several particles can be located at the same point. In the sequel we will say "particle $x_j^i$" instead of saying "particle labeled by $(i,j)$ located at point $(x_j^{i},k-i)$".
\subsection{Geometric jumps}
Let $q=(q_1,\dots,q_r)\in\R_+^r$ and  $\alpha\in (0,1)$ such that $\alpha q_i\in(0,1)$ and $\alpha q_i^{-1}\in(0,1)$ for $i=1,\dots, r$, with $r=[\frac{k+1}{2}]$. Consider two independent families  $$(\xi_j^i(n+\frac{1}{2}))_{i=1,\dots,k,j=1,\dots,[\frac{i+1}{2}];\, n\ge 0}, \quad \textrm{and } \quad (\xi_j^i(n))_{i=1,\dots k,j=1,\dots,[\frac{i+1}{2}];\, n\ge 1},$$   of independent geometric random variables  such that $$\P(\xi^{2i-1}_j(n+\frac{1}{2})=x)=\P(\xi_j^{2i}(n)=x)=(\alpha  q_i^{-1})^x(1-\alpha q_i^{-1}), \quad x\in \N,$$
and 
$$\P(\xi^{2i-1}_j(n)=x)=\P(\xi_j^{2i}(n+\frac{1}{2})=x)=(\alpha q_i)^x(1-\alpha q_i), \quad x\in \N.$$ 
 
The evolution of the particles is given by a process $(X(t))_{t\ge 0}$ on the set $GT_k$ of symplectic Gelfand-Tsetlin patterns. At each time $t\ge 0$, a particle labeled by $(i,j)$ is at point $(X_j^i(t),k-i)$ of $\Z^2$. Particles evolve as follows.  At time $0$ all particles are at zero, i.e. $X(0)=0$.    All particles try to jump to the left at times $n+\frac{1}{2}$ and to the right at times $n$, $n\in \N$.  Suppose that at time $n$, after all particles have jumped, there is one particle at point $(X_j^{i}(n),k-i)$ of $\Z^2$, for $i=1,\dots,k$, $j=1,\dots,[\frac{i+1}{2}]$.  Positions of particles are updated recursively as follows (see figure \ref{figuremodel}).
\\

\noindent \underline{At time $n+1/2$} : 
All particles  try to jump to the left one after another in the lexicographic order pushing the other  particles in order to stay in the set of symplectic  Gelfand-Tsetlin patterns and being blocked by the initial configuration $X(n)$ of the particles:
\begin{itemize}
\item Particle $X_1^1(n)$ tries to move to the left being blocked by $0$, i.e.
$$X_1^1(n+\frac{1}{2})=\max(X_1^1(n)-\xi_1^1(n+\frac{1}{2}),0).$$
\item Particle $X_1^{2}(n)$ tries to jump to the left. It is blocked by $  X_1^1(n)  $. If it is necessary it pushes $X_2^3(n)$ to an intermediate position denoted by $\tilde{X}_2^3(n)$, i.e.
\begin{align*}
X_1^{2}(n+\frac{1}{2})&=\max \big(  X_1^1(n)  ,X_1^2(n)-\xi_1^2(n+\frac{1}{2}) \big)\\
\tilde{X}_2^3(n)&=\min \big(X_2^3(n),X_1^2(n+\frac{1}{2}) \big)
\end{align*}
\item Particle $X_1^3(n)$ tries to move to the left being blocked by $X_1^2(n)$ :
\begin{align*} 
X_1^3(n+\frac{1}{2})=\max \big(X_1^2(n),X_1^3(n)-\xi_1^3(n+\frac{1}{2}) \big).
\end{align*}
 Particle $\tilde{X}_2^3(n)$ tries to move to the left being blocked by $0$, i.e  $$X_2^3(n+\frac{1}{2})=\max(\tilde{X}_2^3(n)-\xi_2^3(n+\frac{1}{2}),0).$$ \end{itemize}
Suppose now that rows 1 through $l-1$ have been updated for some $l>1$. Particles $X_{2}^l(n),\dots,X^l_{[\frac{l+1}{2}]}(n)$ of row $l$ are pushed to intermediate positions  
\begin{align*}
\tilde{X}_i^{l}(n)&=\min\big(X_i^{l}(n),X_{i-1}^{l-1}(n+\frac{1}{2})\big), \, i\in\{2,\dots,[\frac{l+1}{2}]\}.
\end{align*}
 Then particles $X_{1}^l(n),\tilde{X}_{2}^l(n),\dots, \tilde{X}^l_{[\frac{l+1}{2}]}(n)$  try to jump to the left  being blocked as follows by the initial position $X(n)$ of the particles. For $i=1,\dots,[\frac{l+1}{2}]$, 
\begin{align*}
X_i^{l}(n+\frac{1}{2})&=\max\big(X_i^{l-1}(n),\tilde{X}_i^{l}(n)-\xi_i^{l}(n+\frac{1}{2})\big),
\end{align*}
with the convention that $X^{l-1}_{\frac{l+1}{2}}(n)=0$ when $l$ is odd. 

\noindent \underline{At time $n+1$} : All particles try to jump to the right one after another in the lexicographic order pushing  particles in order to stay in the set of symplectic Gelfand-Tsetlin patterns and being blocked by the initial configuration $X(n+\frac{1}{2})$ of the particles.   The first three rows are updated as follows.
\begin{itemize}
\item Particle $X_1^1(n+\frac{1}{2})$ moves to the right pushing $X_1^2(n+\frac{1}{2})$ to an intermediate position $\tilde{X}_1^2(n+\frac{1}{2})$ : 
\begin{align*}  
X_1^1(n+1)&=  X_1^1(n+\frac{1}{2}) +\xi_{1}^1(n+1)  \\
\tilde{X}_1^2(n+\frac{1}{2})&=\max\big({X}_1^2(n+\frac{1}{2}),X_1^1(n+1)\big)
\end{align*}
\item Particle $\tilde{X}_1^2(n+\frac{1}{2})$  jumps to the right   pushing $X_1^3(n+\frac{1}{2})$ to an intermediate position $\tilde{X}_1^3(n+\frac{1}{2})$, i.e.
\begin{align*}
X_1^{2}(n+1)&= \tilde{X}_1^2(n+\frac{1}{2}) + \xi_1^2(n+1)\\
\tilde{X}_1^3(n+\frac{1}{2})&=\max \big(X_1^3(n+\frac{1}{2}),X_1^2(n+1) \big)
\end{align*}
\item Particle $X_2^3(n+\frac{1}{2})$  tries to move to the right being blocked by $X_1^2(n+\frac{1}{2})$. Particle $\tilde{X}_1^3(n+\frac{1}{2})$ moves to the right. That is
\begin{align*} 
X_2^3(n+1)&= \max(X_2^3(n+\frac{1}{2}) +\xi_{2}^3(n+1)) \big\vert,X_1^2(n+\frac{1}{2})) \\
X_1^3(n+1)&=\tilde{X}_1^3(n+\frac{1}{2}) +\xi_1^3(n+1)
\end{align*} 
\end{itemize}
Suppose rows 1 through $l-1$ have been updated for some $l>1$. Then particles of row $l$ are pushed to intermediate positions  
\begin{align*}
\tilde{X}_i^{l}(n+\frac{1}{2})&=\max\big(X_{i}^{l-1}(n+1),X_i^{l}(n+\frac{1}{2})\big), \, i\in\{1,\dots,[\frac{l+1}{2}]\},
\end{align*}
with the convention $X_{\frac{l+1}{2}}^{l-1}(n+1)=0$ when $l$ is odd. Then  particles $\tilde{X}_1^l(n+ \frac{1}{2}),\dots,\tilde{X}_{[\frac{l+1}{2}]}^l(n+ \frac{1}{2})$ try to jump to the right   being blocked by the initial position of the particles as follows. For $i=1,\dots,[\frac{l+1}{2}]$,
\begin{align*}
X_i^{l}(n+1)&=\min\big(X_{i-1}^{l-1}(n+\frac{1}{2}),\tilde{X}_i^{l}(n+\frac{1}{2})+\xi_i^{l}(n+1)\big).
\end{align*}  

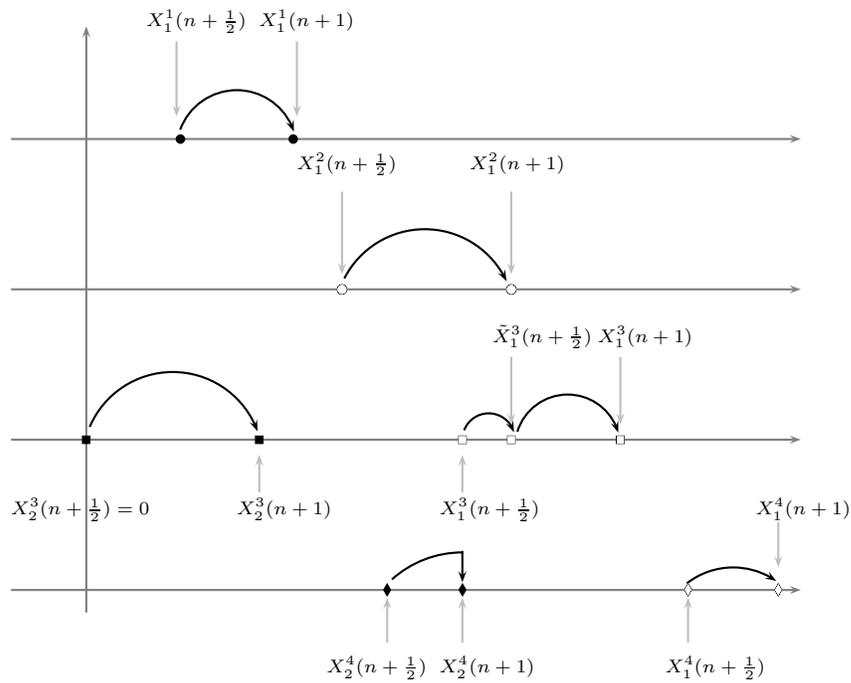
\begin{figure}[h!]
\begin{pspicture}(-9.5,6)(10,-5)
\put(-5,5.5){\underline{Interactions between times n and $n+\frac{1}{2}$}}
\psline[linecolor=gray]{->}(-8,-1)(2.5,-1)
\psline[linecolor=gray]{->}(-8,1)(2.5,1) 
\psline[linecolor=gray]{->}(-8,3)(2.5,3)
\psline[linecolor=gray]{->}(-8,-3)(2.5,-3)
\psline[linecolor=gray]{->}(-7,-3.3)(-7,4.5)

\psline[linecolor=lightgray]{<-}(-4.9,3.3)(-4.9,4.3)
\psline[linecolor=lightgray]{<-}(-5.8,3.3)(-5.8,4.3)
\psline[linecolor=lightgray]{<-}(0.5,-0.8)(0.5,0.1)
\psline[linecolor=lightgray]{<-}(-2,1.2)(-2,2.3)
\psline[linecolor=lightgray]{<-}(-3,-0.7)(-3,0.1)
\psline[linecolor=lightgray]{<-}(-3.6,1.2)(-3.6,2.3)
\psline[linecolor=lightgray]{<-}(2,-2.7)(2,-2.1)
\psline[linecolor=lightgray]{->}(-2,-1.7)(-2,-1.2)
\psline[linecolor=lightgray]{->}(-3.6,-1.7)(-3.6,-1.2) 
\psline[linecolor=lightgray]{->}(-7,-1.7)(-7,-1.2) 
\psline[linecolor=lightgray]{->}(1,-3.7)(1,-3.1) 
\psline[linecolor=lightgray]{->}(0,-3.7)(0,-3.1) 
\psline[linecolor=lightgray]{->}(-2,-3.7)(-2,-3.1) 
\psline[linecolor=lightgray]{->}(-3,-3.7)(-3,-3.1)

\put(-6.5,4.5){\scriptsize{$X^{1}_1(n+\frac{1}{2})$}}
\put(-5,4.5){\scriptsize{$X^1_{1}(n)$}}
\put(1.7,-2){\scriptsize{$X^4_{1}(n)$}}
\put(-0.3,-4.1){\scriptsize{$X^4_{2}(n)$}}
\put(-2.3,2.6){\scriptsize{$X^2_{1}(n)$}}
\put(0.2,0.3){\scriptsize{$X^3_{1}(n)$}}
\put(-3.3,0.3){\scriptsize{$X^3_{2}(n)$}}
\put(-4.2,2.6){\scriptsize{$X^{2}_1(n+\frac{1}{2})$}}
\put(-2.3,-2){\scriptsize{$X^{3}_1(n+\frac{1}{2})$}}
\put(-4,-2){\scriptsize{$\tilde{X}_{2}^3(n)$}}
\put(-8,-2){\scriptsize{$X_{2}^3(n+\frac{1}{2})=0$}}
\put(0.7,-4.1){\scriptsize{$X_{1}^4(n+\frac{1}{2})$}}
\put(-2.3,-4.1){\scriptsize{$\tilde{X}_{2}^4(n)$}}
\put(-3.8,-4.1){\scriptsize{${X}_{2}^4(n+\frac{1}{2})$}}

\psarc{->}(-5.3,2.85){0.5}{34}{150}
\psarc{->}(-2.8,0.8){0.8}{23}{159}
\psarc{-}(-5.6,-2){2.2}{31}{129.5}
\psarc{-}(-1,-2){1.8}{37}{124}
\psarc{->}(-3.3,-0.9){0.3}{0}{180}
\psarc{->}(1.5,-3){0.5}{14}{168}
\psarc{->}(-1,-3.9){1.4}{48}{135}
\psarc{-}(-3,-3.9){1.4}{48}{90}

\psline{->}(-7,-0.29)(-7,-0.9) 
\psline{->}(-2,-0.5)(-2,-0.9) 
\psline{->}(-3,-2.5)(-3,-2.9) 

\psdots[dotstyle=o,dotsize=4pt](-3.6,1)(-2,1)
\psdots[dotstyle=square](0.5,-1)(-2,-1)
 \psdots(-4.9,3)(-5.8,3)
 \psdots[dotstyle=square*](-3.6,-1)(-3,-1)(-7,-1)
 \psdots[dotstyle=diamond](2,-3)(1,-3)
 \psdots[dotstyle=diamond*](-2,-3)(0,-3)(-3,-3)

\end{pspicture} 
\begin{pspicture}(-9.5,5.5)(10,-4)

\put(-5,5.5){\underline{Interactions between times $n+\frac{1}{2}$ and $n+1$}}
\psline[linecolor=gray]{->}(-8,-1)(2.5,-1)
\psline[linecolor=gray]{->}(-8,1)(2.5,1) 
\psline[linecolor=gray]{->}(-8,3)(2.5,3)
\psline[linecolor=gray]{->}(-8,-3)(2.5,-3)
\psline[linecolor=gray]{->}(-7,-3.3)(-7,4.5)

\psline[linecolor=lightgray]{<-}(-3.6,1.2)(-3.6,2.3)
\psline[linecolor=lightgray]{<-}(-1.35,1.2)(-1.35,2.3)
\psline[linecolor=lightgray]{<-}(0.1,-0.8)(0.1,0.1)
\psline[linecolor=lightgray]{<-}(-1.35,-0.8)(-1.35,0.1)
\psline[linecolor=lightgray]{->}(1,-3.7)(1,-3.1) 
\psline[linecolor=lightgray]{->}(-2,-3.7)(-2,-3.1) 
\psline[linecolor=lightgray]{->}(-3,-3.7)(-3,-3.1) 
\psline[linecolor=lightgray]{<-}(-5.8,3.3)(-5.8,4.3)
\psline[linecolor=lightgray]{<-}(-4.2,3.3)(-4.2,4.3)
\psline[linecolor=lightgray]{<-}(2.2,-2.7)(2.2,-2.1)
\psline[linecolor=lightgray]{->}(-2,-1.7)(-2,-1.2)
\psline[linecolor=lightgray]{->}(-4.7,-1.7)(-4.7,-1.2)

\put(-6.2,4.5){\scriptsize{$X^{1}_1(n+\frac{1}{2})$}}
\put(-4.7,4.5){\scriptsize{$X^{1}_1(n+1)$}}
\put(-1.9,2.6){\scriptsize{$X^2_{1}(n+1)$}}
\put(-5,-2){\scriptsize{$X_{2}^3(n+1)$}}
\put(-2.3,-4.1){\scriptsize{${X}_{2}^4(n+1)$}}
\put(-0.2,0.3){\scriptsize{$X^3_{1}(n+1)$}} 
\put(-1.6,0.3){\scriptsize{$\tilde{X}^3_{1}(n+\frac{1}{2})$}} 

\put(1.9,-2){\scriptsize{$X^4_{1}(n+1)$}}
\put(-4.2,2.6){\scriptsize{$X^{2}_1(n+\frac{1}{2})$}}
\put(-2.3,-2){\scriptsize{$X^{3}_1(n+\frac{1}{2})$}}
\put(0.7,-4.1){\scriptsize{$X_{1}^4(n+\frac{1}{2})$}}
\put(-3.8,-4.1){\scriptsize{${X}_{2}^4(n+\frac{1}{2})$}}
\put(-8,-2){\scriptsize{$X_{2}^3(n+\frac{1}{2})=0$}}

\psarc{<-}(-5,2.85){0.8}{18}{160}
\psarc{<-}(-2.5,0.6){1.2}{25}{153}
\psarc{<-}(-5.85,-1.3){1.2}{20}{157}
\psarc{<-}(-1.65,-1){0.35}{16}{159}
\psarc{<-}(-0.6,-1.1){0.7}{16}{163}
\psarc{-}(-2,-3.9){1.4}{90}{132} 
\psarc{<-}(1.6,-3.7){1}{55}{128} 
  
\psline{->}(-2,-2.48)(-2,-2.9) 

\psdots[dotstyle=o,dotsize=4pt](-3.6,1)(-1.35,1)
\psdots[dotstyle=square](-2,-1)(-1.35,-1)(0.1,-1)
 \psdots(-4.25,3)(-5.75,3)
 \psdots[dotstyle=square*](-7,-1)(-4.7,-1)
 \psdots[dotstyle=diamond](1,-3)(2.2,-3)
 \psdots[dotstyle=diamond*](-3,-3)(-2,-3)

\end{pspicture} 

\caption{An example of blocking and pushing interactions between times $n$ and $n+1$ for $k=4$. Different kinds of dots represent different particles.}
\label{figuremodel}
\end{figure}

\subsection{Exponential waiting times}\label{Poisson}
The interacting particle model described now has been introduced in \cite{WarrenWindridge}. In this model particles evolve on $\Z^2$ and  jump on their own volition by one  rightwards or  leftwards  after an exponentially distributed waiting time. The evolution of the particles is described by a random process $(Y(t))_{t\ge 0}$ on  $GT_{k}$. As in the previous model, at time $t\ge 0$ there is one particle labeled by $(i,j)$ at   point $(Y^{i}_j(t),k-i)$ of the integer lattice, for $i=1,\dots,k$, $j=1,\dots,[\frac{i+1}{2}]$. Particle labeled by $(2i,j)$ tries to jump to the left by one after an exponentially distributed waiting time with mean $q_i$ or to the right by one after an exponentially distributed waiting time with mean $q_i^{-1}$. Particle labeled by $(2i-1,j)$ tries to jump to the left by one after an exponentially distributed waiting time with mean $q_i^{-1}$ or to the right by one after an exponentially distributed waiting time with mean $q_i$. Waiting times are all independent. When a particle tries to jump, all particles are pushed and blocked according to the same rules as previously: particles above push and block particles below. That is if particle labeled by $(i,j)$ wants to jump to the right at time $t\ge 0$ then
\begin{enumerate}
\item if $i,j\ge 2$ and $Y_j^i(t^-)=Y_{j-1}^{i-1}(t^-)$ then particles don't move and $Y(t)=Y(t^-)$. 
\item else particles $(i,j),(i+1,j),\dots,(i+l,j)$ jump  to the right  by one for $l$  the largest integer such that $Y_j^{i+l}(t^-)=Y_j^i(t^-)$ i.e. $$Y_j^{i}(t)=Y_j^{i}(t^-)+1,\dots,Y_j^{i+l}(t)=Y_j^{i+l}(t^-)+1.$$ 
\end{enumerate}
If particle labeled by $(i,j)$ wants to jump to the left at time $t\ge 0$ then
\begin{enumerate}
\item if $i$ is odd, $j=(i+1)/2$ and $Y_j^{i}(t^-)=0$ then particle labeled by $(i,j)$ doesn't move.
\item if $i$ is odd, $j=(i+1)/2$ and $Y_j^{i}(t^-)\ge 1$ then $Y_j^i(t)=Y_j^i(t^-)-1$.
\item if $i$ is even or $j\ne (i+1)/2$, and $Y_j^i(t^-)=Y_j^{i-1}(t^-)$ then particles don't move.
\item if $i$ is even or $j\ne (i+1)/2$, and $Y_j^i(t^-)>Y_j^{i-1}(t^-)$ then particles $(i,j),(i+1,j+1),\dots,(i+l,j+l)$ jump  to the left  by one for $l$  the largest integer such that $Y_{j+l}^{i+l}(t^-)=Y_j^i(t^-).$ Thus $$Y_j^{i}(t)=Y_j^{i}(t^-)-1,\dots,Y_{j+l}^{i+l}(t)=Y_{j+l}^{i+l}(t^-)-1.$$ 
\end{enumerate}

Actually, process $(Y(t),t\ge 0)$ is obtained by letting $\alpha$ go to zero in the previous model. More precisely we get the following proposition. 
 \begin{prop}   \label{conv0} The process $(X([\alpha^{-1}t]),t\ge 0)$ converges in the sense of finite-dimensional distributions towards the process $(Y(t),t\ge 0)$ as $\alpha$ goes to zero.
\end{prop}
\begin{proof} 
The proposition is obtained  by replacing $q$ by $\alpha$ in Lemma 8.9 of \cite{DefosseuxSO}.
\end{proof}

\section{A Pieri type formula for the symplectic group}

Let $r$ be a positive integer. One recalls some usual properties of  the finite dimensional representations of the compact symplectic  group $Sp_{2r}$  (see for instance \cite{Knapp} for more details).  The set of finite dimensional representations of $Sp_{2r}$ is indexed by the set  
$$\mathcal{W}_{2r}=\{\lambda\in \R^r:\lambda_r\in \N, \lambda_i-\lambda_{i+1}\in \N, i=1,\dots,r-1\}.$$ For $\lambda\in \mathcal{W}_{2r}$, using standard notations,  we denote by $V_\lambda$ the so called irreducible representation with highest weight $\lambda$ of $Sp_{2r}$.   

  Let $m$ be an integer and $\lambda$ an  element of  $\mathcal{W}_{2r}$. Consider the irreducible representations $V_\lambda$ and $V_{\gamma_m}$ of  $Sp_{2r}$, with $\gamma_m=(m,0,\cdots,0)$. The decomposition of the tensor product $V_\lambda\otimes V_{\gamma_m}$ into irreducible components is given by a Pieri-type formula for the symplectic group. It has been recalled in \cite{Defosseux}.  We have 
    \begin{align}\label{MultC} V_{\lambda}\otimes V_{\gamma_m}=\oplus_{\beta}M_{\lambda,\gamma_m}(\beta) V_{\beta},
  \end{align}
  where the direct sum is over all  $\beta\in \mathcal{W}_{2r}$ such that  
   there exists  $c\in \mathcal{W}_{2r}$ which satisfies $$\left\{
    \begin{array}{l}
      c\preceq \lambda , \quad c\preceq \beta \\
        \\
      \sum_{i=1}^{r}(\lambda_i-c_i+ \beta_i-c_i)=m.
    \end{array}
\right. $$  In addition, the  multiplicity $M_{\lambda,\gamma_m}(\beta)$ of the irreducible representation with highest weight $\beta $ is the   number of $c\in\mathcal{W}_{2r}$ satisfying these relations. Note  that $V_{\lambda}\otimes V_{\gamma_m}$ is not free  multiplicity if $m\notin \{0,1\}$.

\section{Markov kernels}
 Since for $\lambda\in \mathcal{W}_{2r}$ the Schur function $s^{2r}_\lambda$ is the character of the irreducible representation $V_\lambda$, decomposition (\ref{MultC}) implies  
\[
s_\lambda^{2r}(q)s^{2r}_{\gamma_m}(q)=\sum_{\beta\in \mathcal{W}_{2r}} M_{\lambda,\gamma_m}(\beta)s^{2r}_\beta(q).
\]
Thus one defines a family $(\mu_m)_{m\ge 0}$ of Markov kernels on $\mathcal{W}_{2r}$ by letting 
$$\mu_m(\lambda,\beta)=\frac{s^{2r}_\beta(q)}{s^{2r}_\lambda(q) s^{2r}_{\gamma_m}(q)} M_{\lambda,\gamma_m}(\beta),$$
for $\lambda,\beta\in \mathcal{W}_{2r}$ and $m\ge 0$.  Let $\xi_1,\dots ,\xi_{2r}$ be independent random variable such that  $\xi_1,\xi_2,\dots ,\xi_{r}$ are geometric random variables with  respective parameters $\alpha q_1,\dots,\alpha q_{r}$ and $\xi_r,\xi_{r+1},\dots ,\xi_{2r}$ are geometric random variables with  respective parameters  $\alpha q^{-1}_1,\dots,\alpha q^{-1}_{r}$.
Consider a random variable $T$ on $\N$  defined by
$$T=\sum_{i=1}^{2r}\xi_i.$$

\begin{lem} \label{nu} The law of $T$ is a measure $\nu$ on $\N$ defined by
$$\nu(m)=\alpha^m a(q)s^{2r}_{\gamma_m}(q),\quad m\in \N,$$
where
\[ 
a(q)=\prod_{i=1}^{r}(1-\alpha q_i)(1-\alpha q_i^{-1}).
\] 
\end{lem} 
\begin{proof}
The lemma follows from straightforward computations.
\end{proof}
Lemma \ref{nu} implies in particular that the measure $\nu$ is a probability measure. Thus one defines a Markov kernel $P_{2r}$ on $\mathcal{W}_{2r}$ by letting 
\begin{align*} 
P_{2r}(\lambda,\beta)=\sum_{m=0}^{+\infty}\mu_m(\lambda,\beta)\nu(m),
\end{align*}
for $\lambda,\beta\in \mathcal{W}_{2r}$. 
 \begin{prop}  \label{explicitP2r} For $\lambda,\beta\in \mathcal{W}_{2r}$, 
 \begin{align}\label{P2r}
P_{2r}(\lambda,\beta)=\sum_{c\in \mathcal{W}_{2r}: c\preceq \lambda,\beta}a(q)\frac{s^{2r}_\beta(q)}{s^{2r}_\lambda(q)}\alpha^{\sum_{i=1}^{r}(\lambda_i+\beta_i-2c_i)}.
\end{align}
 \end{prop}
 \begin{proof} The proposition follows immediately from the tensor product rules  recalled  for the decomposition (\ref{MultC}).
 \end{proof}
 
We let $\mathcal{W}_{2r-1}=\mathcal{W}_{2r}$.  For $c_0,\lambda,c,\beta\in \mathcal{W}_k$, we let 
\begin{align}\label{S2r}
S_{k}((c_0,\lambda),(c,\beta))=a(q)\frac{s^{k}_\beta(q)}{s^{k}_\lambda(q)}\alpha^{\sum_{i=1}^{r}(\lambda_i+\beta_i-2c_i)}1_{c\preceq \lambda, \beta},
\end{align}
when $k=2r$ and 
\begin{align}\label{S2r-1}
S_{k}((c_0,\lambda),(c,\beta))=\tilde{a}(q)\frac{s^{k}_\beta(q)}{s^{k}_\lambda(q)}\alpha^{\sum_{i=1}^{r}(\lambda_i+\beta_i-2c_i)}((1-\alpha q_r^{-1})1_{c_r>0}+1_{c_r=0})1_{c\preceq \lambda,\beta},
\end{align}
when $k=2r-1$, with 
\[ 
\tilde{a}(q)=(1-\alpha q_r)\prod_{i=1}^{r-1}(1-\alpha q_i)(1-\alpha q_i^{-1}).
\] 
The main purpose of this paper is to show that $S_k$ describes the evolution of the $k^{\textrm{th}}$ row of the random symplectic Gelfand-Tsetlin patterns $(X(t),t\ge 0)$. Note that Proposition \ref{explicitP2r} ensures that $S_{2r}$ defines a Markov kernel on $\mathcal{W}_{2r}\times \mathcal{W}_{2r}$. There isn't such an argument for $S_{2r-1}$.  Anyway as $\Lambda_k$ and $Q_k$ defined in section \ref{proofs} are Markov kernels, Proposition \ref{intertwining} ensures that $S_{k}$ is a Markov kernel in both the odd and the even cases. Thus one defines also a Markov kernel $P_{2r-1}$ on $\mathcal{W}_{2r-1}$ by letting
 \begin{align}\label{P2r-1}
P_{2r-1}(\lambda,\beta)=\sum_{c\in  \mathcal{W}_{2r}: \,c\preceq \lambda,\beta}\tilde{a}(q)\frac{s^{2r-1}_\beta(q)}{s^{2r-1}_\lambda(q)}\alpha^{\sum_{i=1}^{r}(\lambda_i+\beta_i-2c_i)} ((1-\alpha q_r^{-1})1_{c_r>0}+1_{c_r=0}).
\end{align}
Actually $P_{2r-1}(\lambda,.)$ is the image of the measure $S_{2r-1}((c_0,\lambda),(.,.))$, for any arbitrary $c_0\in \mathcal W_{2r}$, by the map 
$$(x,y)\in \mathcal{W}_{2r}\times \mathcal{W}_{2r}\mapsto y\in \mathcal{W}_{2r}.$$

We'll  see that the image measure $P_k$ describes the evolution of the $k^{\textrm{th}}$ row of the random symplectic Gelfand-Tsetlin patterns observed at integer times. This Markov kernel is relevant for the understanding of the relation between the particle model and the random matrix model of the next section.
  \section{Random matrices}
We denote by $\H$ the set of quaternions. For us, a quaternion is just a $2 \times 2$ 
  matrix $Z$ with complex entries which can be written as 
  $$ Z=
\begin{pmatrix}
  a& b   \\
 -\bar b &  \bar a   
\end{pmatrix},
  $$
  where $a,b\in \C$. Its conjugate  $Z^*$  is the usual adjoint of the complex matrix $Z$. Let us denote by ${\mathcal M}_{r,m}$ the  real vector space of $r \times m$   matrices with entries in $\H$ and by ${\mathcal P}_{r}$   the set  of $r\times r$ Hermitian matrices with entries in $i\H$. Since a matrix in  ${\mathcal P}_{r}$ is a $2r\times 2r$ Hermitian complex matrix, it has real eigenvalues $\lambda_1 \geq \lambda_2 \geq \cdots \geq \lambda_{2r}$. Moreover $\lambda_{2r-i+1}= -\lambda_{i}$, for $i=1,\cdots,2r$.   We put on ${\mathcal M}_{r,m}$ the Euclidean structure defined by the scalar product, 
$$\langle M, N \rangle = \mbox{tr}(MN^*), \quad M,N\in {{\mathcal M}_{r,m}}.$$
Let $\mathcal{C}_r$ be the subset of $\R^r$  defined by
$$\mathcal{C}_r=\{x\in \R^r : x_1>\dots>x_r>0\}.$$

Theorem 4.5 of \cite{Defosseux} and Proposition 4.8 of \cite{Defosseux} imply the following proposition. 
\begin{prop}  \label{eigenvalues}  Let $r$ be a positive integer and $(M(n),n\ge 0)$, be a discrete process  on $\mathcal{P}_{r}$ defined by $$M(n)=\sum_{l=1}^{n}Y_l \begin{pmatrix} 1 &  0  \\
0 &  -1       
\end{pmatrix}Y_l^*,$$ where the $Y_l$'s are independent standard Gaussian variables in $\mathcal{M}_{r,1}$. For $n\in \N$, let  $\Lambda_1(n),\cdots,\Lambda_r(n)$ be the $r$ largest eigenvalues of $M(n)$ such that 
\[
\Lambda_1(n)\ge\cdots\ge\Lambda_r(n).
\] Then the process $(\Lambda(n),n\ge 0),$ is a Markov process with a transition densities $p_r$ defined by
\[ 
 p_r(x,y)=\frac{d_r(y)}{d_r(x)} I(x,y),\quad x,y\in \CC_r , 
 \]
 where 
 \[
 I(x,y)= 
          \int_{\mathbb{R}_+^{r}} 1_{\{x,y\succ z\}}e^{-\sum_{i=1}^{r}(x_i+y_i-2z_i)}\,dz,         
\]
and 
\[
d_r(x)=\prod_{1\le i<j\le r}(x_i^2-x_j^2)\prod_{i=1}^r x_i, \quad x,y\in \mathcal{C}_r.
\]
\end{prop}

  \section{results}
  We have  introduced a Markov process $(X(t),t\ge 0)$ on the set of symplectic Gelfand-Tsetlin patterns. We will show that if only  one row of the pattern is considered by itself, it found to be a Markov process too. Similar results have been proved in  \cite{WarrenWindridge}. The main specificity of our model is that coordinates of particles are not upgraded in the same way at every times. We are mainly interested in the process $(X^k(n),n\ge 0)$ but intermediate states $(X^k(n+\frac{1}{2}),n\ge 0)$  are considered for the proofs.  Actually these intermediate states come from the fact that the tensor product of (\ref{MultC}) is not free-multiplicity.  We let $Z^k(n)=(X^{k}(n-\frac{1}{2}),X^{k}(n))$, for $n\in \N$, with  $X^{k}(-\frac{1}{2})=0$. 
  \begin{theo}\label{theoprel} The process $(Z^k(n),n\ge 0)$  is a Markov process on $\mathcal{W}_{k}\times \mathcal{W}_{k} $ with transition kernel $S_k$.  
  \end{theo} 
   If $P_k$ is the Markov kernel defined in (\ref{P2r}) and (\ref{P2r-1}) then Theorem \ref{theoprel} implies immediately  the following theorem which is our main result. 
\begin{theo}\label{maintheo}
The process $(X^{k}(n))_{n\ge 0}$ is a Markov process on $\mathcal{W}_{k}$ with transition kernel $P_k$.  
\end{theo} 
Convergence stated in Proposition \ref{conv0} and Theorem \ref{maintheo}   lead to the following corollary, which is exactly Theorem 2.3 of \cite{WarrenWindridge}.  Let us denote $e_1,\dots,e_{[\frac{k+1}{2}]}$ the standard basis of $\R^{[\frac{k+1}{2}]}$.
\begin{cor} 
The process $(Y^k(t),t\ge 0)$ is a Markov process with infinitesimal generator defined by
$$A(x,y)=\frac{s_{y}^k(q)}{s_x^k(q)}1_{y\in \mathcal{W}_k},$$ 
for $x\in \mathcal{W}_k, y\in \{x+e_1,\dots,x+e_{[\frac{k+1}{2}]},x-e_1,\dots,x-e_{[\frac{k+1}{2}]}\}.$
\end{cor}
\begin{proof}
Theorem \ref{maintheo} and Lemma 2.21 of \cite{BorodinFerrari} implies that the process $$(X^k([\alpha^{-1}t]),t\ge 0)$$ converges towards a Markov process with infinitesimal  generator equal to $A$ as $\alpha$ goes to zero. The convergence stated in  Proposition \ref{conv0}  achieves the proof.
\end{proof}
If $(\Lambda(n),n\ge 0)$ is the process of eigenvalues considered in Proposition \ref{eigenvalues} then the following corollary holds.
\begin{cor} \label{coroMat} Letting  $q_i=1$ for $i=1,\dots,r$, the process $((1-\alpha)X^{2r}(n), n\ge 1)$ converges in distribution towards the process of eigenvalues $(\Lambda(n), n\ge 1)$ as $\alpha$ goes to one.  
\end{cor}
\begin{proof}
 The Weyl dimension formula  (see Knapp \cite{Knapp}, Thm V.5.84)  for the symplectic groups gives  $$s_\lambda^{2r}({\bf 1})=\prod_{1\le i<j\le r} \frac{(\lambda_i-\lambda_j+j-i)(\lambda_i+\lambda_j+2n+2-j-i)}{(j-i)(2n+2-j-i)}\prod_{i=1}^r\frac{\lambda_i+n+1-i}{n+1-i}.$$ 
Thus the corollary  follows immediately from Theorem \ref{maintheo} and Proposition \ref{eigenvalues}.
\end{proof}

\section{proofs}\label{proofs}
The proof of Theorem  \ref{theoprel} rests on the same ingredients as the proof of Proposition 8.8 of  \cite{DefosseuxSO}. It will follow from an intertwining property stated in Proposition \ref{intertwining} and an application of a  criterion established in \cite{PitmanRogers} by Pitman and Rogers who give a simple condition sufficient to ensure that a function of a Markov process is again a Markov process.
For $\lambda\in \mathcal{W}_{k}$ we consider the measure $M_\lambda$ on $GT_k(\lambda)$ defined by 
$$M_\lambda=\sum_{x\in GT_k(\lambda)}\frac{w_k(x)}{s^k_\lambda(q)}\delta_x,$$
where $\delta_x$ is the dirac measure at $x$, 
and the measure $m_\lambda$ defined as the image of  the measure $M_\lambda$ by the map $x\in GT_k(\lambda)\mapsto x^{k-1}\in \mathcal{W}_{k-1}$, i.e  
 \[
m_\lambda=\sum_{\beta\in \mathcal{W}_{k-1}: \beta\preceq \lambda} q_r^{\vert\beta\vert -\vert \lambda\vert }\frac{s^{k-1}_\beta(q)}{s^k_ \lambda(q)}\delta_\beta, 
\]
when $k=2r$, and
\[
m_\lambda=\sum_{\beta\in \mathcal{W}_{k-1}: \beta\preceq \lambda} q_{r}^{\vert \lambda\vert -\vert\beta\vert }\frac{s^{k-1}_\beta(\tilde{q})}{s^k_ \lambda(q)}\delta_\beta, 
\]
when $k=2r-1$, with $\tilde{q}=(q_1,\dots,q_{r-1})$. One defines a kernel $\Lambda_k$ form $\mathcal{W}_k\times\mathcal{W}_k$ to $\mathcal{W}_{k-1}\times\mathcal{W}_k\times\mathcal{W}_k$ by letting
\[
\Lambda_k((c,\lambda),(\beta,c',\lambda'))=m_\lambda(\beta)1_{c\preceq \lambda}1_{c=c',\lambda=\lambda'},
\]
 for $c,\lambda,c',\lambda'\in \mathcal{W}_k$, $\beta\in \mathcal{W}_{k-1}$.\\
 
 As $(Z^k(n),n\ge 0)$ is conditionally independent of $(Z^i(t),t\ge 0,i=1,\dots,k-2)$ given $(Z^{k-1}(t),t\ge 0)$, Theorem \ref{theoprel} may be proved by induction on $k$. The theorem is true for $k=1$. Suppose that the process $(Z^{k-1}(n),n\ge 0)$ is Markovian with transition kernel $S_{k-1}$. The dynamic of the model implies that $$(Z^{k-1}(n),Z^{k}(n), n\ge 0)$$ is also Markovian. As for any $\lambda\in \mathcal{W}_{k-1}$, $$S_{k-1}((c_0,\lambda),(.,.))$$ doesn't depend on $c_0$, it implies that the process $$((X^{k-1}(n),Z^k(n)), n\ge 0)$$ is also Markovian.  Let us denote by $Q_k$ its transition kernel. Proposition \ref{intertwining} claims that it satisfies an intertwining property, which implies, using the Rogers and Pitman criterion of \cite{PitmanRogers},  that the process $(Z^{k}(n),n\ge 0)$ is Markovian with transition kernel $S_k$. Thus the theorem is true for the integer $k$.
\begin{prop} \label{intertwining}
\[
\Lambda_k Q_k = S_k \Lambda_k
\]
\end{prop}
\begin{proof} 

For $x,y,z\in \mathcal{W}_{k}$, we let 
$$S_{k}(x,(y,z))=S_{k}((c,x),(y,z)),$$
where $c$ is any vector of $\mathcal{W}_{k}$ such that $c\preceq x$. 
Let $\xi^+$ and $\xi^-$ be two geometric random variables with respective parameters $\alpha q_k$ and $\alpha q_k^{-1}$. For $a,b\in \R_+$ we denote by $$\overset{ b\leftarrow }{P_{k}}(a,.)$$  the law of $\max(a-\xi^+,b)$ when $k$ is even and the law of $\max(a-\xi^-,b)$ when $k$ is odd. We denote by $$\overset{ \rightarrow b }{P_{k}}(a,.)$$  the law of $\min(a+\xi^-,b)$ when $k$ is even and the law of $\min(a+\xi^+,b)$ when $k$ is odd.    We have  for $(u,z,y),(x,z',y')\in  \mathcal{W}_{k-1}\times \mathcal{W}_{k}\times  \mathcal{W}_{k}$ such that $u,z\preceq y$ and $x,z'\preceq y'$  
\begin{align}\label{Qlodd} 
Q_{k}((u,z,y),(x,z',y'))=\sum_{v\in \mathcal{W}_{k-1}}S_{k-1}(u,(v,x)) & \prod_{i=1}^{[ \frac{k+1}{2}]}\overset{u_{i}\leftarrow}{P_{k}}(y_{i}\wedge v_{i-1},z'_{i}) \nonumber \\
&\times\prod_{i=0}^{[\frac{k+1}{2}]-1}\overset{ \rightarrow v_{i}}{P_{k}}(z'_{i+1}\vee x_{i+1},y'_{i+1}), 
\end{align}
with  the conventions that $x_{[\frac{k+1}{2}]}=u_{[\frac{k+1}{2}]}=0$ when $k$ is odd, $v_0=+\infty$  in the odd and the even cases and the sum runing over $v\in  \mathcal{W}_{k-1}$ such that $v_i\in\{y'_{i+1},\dots,x_i\wedge z'_i\}$, for $i\in\{1,\dots,[\frac{k+1}{2}]-1\}$. Then we write
 
 \begin{align*}
L_{k}Q_{k}((z,y),(x,z',y'))=\sum_{u\in \mathcal{W}_{k},v\in \mathcal{W}_{k-1}}L_{k}((z,y)&,(u,z,y))S_{k-1}(u,(v,x)) \nonumber \\ & \times\prod_{i=1}^{[\frac{k+1}{2}]}\overset{u_{i}\leftarrow}{P_{k}}(y_{i}\wedge v_{i-1},z'_{i}) \nonumber \\ 
&\quad \times\prod_{i=0}^{[\frac{k+1}{2}]-1}\overset{ \rightarrow v_{i}}{P_{k}}(z'_{i+1}\vee x_{i+1},y'_{i+1}).\end{align*}

We get  the intertwining summing over $u$ first and over $v$ after, using respectively identities (4) and (5) of Lemma 8.3 of  \cite{DefosseuxSO}.

\end{proof}

\end{document}